\documentclass{article}
\usepackage{fullpage}
\usepackage{amsmath}
\usepackage{amsthm}
\usepackage{amssymb}
\usepackage{url}
\usepackage{algorithm}
\usepackage{algorithmic}
\usepackage{graphicx}
\usepackage{verbatim}
\usepackage{url}
\usepackage{graphicx}
\usepackage{tikz}
\usetikzlibrary{calc}
\usetikzlibrary{patterns}
\tikzstyle{mybox} = [draw=black, fill=white,  thick,
    rectangle, inner sep=10pt, inner ysep=20pt]
\tikzstyle{mybox} = [draw=black, fill=white,  thick,
    rectangle, inner sep=2pt, inner ysep=2pt]

\newtheorem{thm}{Theorem}
\newtheorem{lemma}{Lemma}
\newtheorem{cor}{Corollary}
\newtheorem{prop}{Proposition}
\theoremstyle{definition}
\newtheorem{definition}{Definition}
\newtheorem{remark}{Remark}

\begin{document}

\title{A Triangle Algorithm for Semidefinite Version of Convex Hull Membership Problem}
\author{Bahman Kalantari\\
Department of Computer Science, Rutgers University, NJ\\
kalantari@cs.rutgers.edu
}
\date{}
\maketitle

\begin{abstract}
Given a subset $\mathbf{S}=\{A_1, \dots, A_m\}$ of $\mathbb{S}^n$, the set of  $n \times n$ real symmetric matrices, we define its {\it spectrahull} as the set $SH(\mathbf{S}) = \{p(X) \equiv (Tr(A_1 X), \dots, Tr(A_m X))^T : X \in \mathbf{\Delta}_n\}$, where ${\bf \Delta}_n$ is the {\it spectraplex},
$\{ X \in \mathbb{S}^n : Tr(X)=1, X \succeq 0 \}$.  We let {\it spectrahull membership} (SHM) to be the problem of testing if a given $b \in \mathbb{R}^m$ lies in
$SH(\mathbf{S})$.  On the one hand when $A_i$'s are diagonal matrices, SHM reduces to the {\it convex hull membership} (CHM), a fundamental problem in LP. On the other hand, a bounded  SDP feasibility is reducible to SHM.  By building on the {\it Triangle Algorithm} (TA) \cite{kalchar,kalsep}, developed for CHM and its generalization,
we prove $b \in SH(\mathbf{S})$ if and only if for each $X  \in \mathbf{\Delta}_n$ there exists $V \in \mathbf{\Delta}_n$, called {\it pivot}, such that $\Vert p(X)- p(V)\Vert \geq \Vert b-p(V)\Vert$. Using this we design a TA for SHM, where given $\varepsilon$,  in $O(1/\varepsilon^2)$ iterations it either computes a hyperplane separating $b$ from $SH(\mathbf{S})$, or $X_\varepsilon \in \mathbf{\Delta}_n$ such that $\Vert p(X_\varepsilon) - b \Vert \leq \varepsilon R$,  $R$  maximum error over $\mathbf{\Delta}_n$. Under certain conditions iteration complexity improves to $O(1/\varepsilon)$ or even $O(\ln 1/\varepsilon)$.  The worst-case complexity of each iteration is $O(mn^2)$, plus testing the existence of a pivot, shown to be equivalent to estimating the least eigenvalue of a symmetric matrix. This together with a semidefinite  version of Carath\'eodory theorem allow implementing TA as if solving a CHM, resorting to the {\it power method} only as needed, thereby improving the complexity of iterations.  The proposed Triangle Algorithm for SHM is simple, practical and applicable to general SDP feasibility and optimization. Also, it extends to a spectral analogue of SVM for separation of two spectrahulls.
\end{abstract}





\section{Introduction} \label{sec1}
The {\it convex hull membership} problem (CHM) is the purest and simplest form of a linear programming feasibility problem.  Formally, given a subset $S= \{v_1, \dots, v_n\} \subset \mathbb{R} ^m$ and a distinguished point $p_\circ \in \mathbb{R} ^m$,  CHM is the problem of testing if $p_\circ  \in conv(S)$, the convex hull of $S$.  If $p_\circ \in conv(S)$, a representation as a convex combination of the points in $S$ is to be given, otherwise a certificate that $p_\circ \not \in conv(S)$.  In practice, we either compute an approximate feasible point, or provide a hyperplane that separates $p_\circ$ from $conv(S)$. By virtue of LP duality,
CHM is essentially equivalent to the general LP, see e.g. \cite{jinkal}. Aside from LP, CHM is a fundamental problem in computational geometry, machine learning, statistics and more. The task of solving CHM, either directly or indirectly, has been the focus of the pioneering polynomial time algorithms for LP by Khachiyan  \cite{kha79} and by Karmarkar \cite{kar84}.  The homogeneous case of CHM  has given rise to the  {\it matrix scaling dualities}, \cite{kalcan}, based on which \cite{KK} states a simple algorithm for either solving a homogeneous CHM, or the {\it quasi doubly-stochastic diagonal scaling} for an associated  positive semidefinite symmetric matrix.  An important application of CHM  in computational geometry and  machine learning is using it as an oracle in solving the {\it irredundancy problem}, the problem of computing all vertices of  $conv(S)$, see e.g. \cite{AKZ}.

CHM can of course be solved via any LP-based algorithm.  However, there are merits in designing specialized algorithms.  Polynomial time algorithms for CHM depend polynomially in $m,n$ and $\ln 1/\varepsilon$, where $\varepsilon$ is the desired accuracy of approximate solution produced. However,  for rational inputs these algorithms are capable of deciding, in polynomial time, the feasibility of CHM.
When the number of points, $n$, and their dimension, $m$, are large, polynomial time algorithms become prohibitive because of high degree of dependence on $n$ and $m$.  One can argue, be it from the practical, theoretical or intellectual points of view, that it is equally important to explore algorithms that run with low degree polynomial dependence on $n$ and $m$, but  polynomial dependence on $1/\varepsilon$.  These considerations have led to the study of {\it fully polynomial time approximation schemes} (FPTAS) for CHM.

Beyond LP, the next class of optimization problems that have found numerous applications are {\it Semidefinite Programming} (SDP).  The relevance of FPTAS in SDP becomes even more pronounced because while LP feasibility is polynomially decidable over rational inputs, SDP feasibility may be undecidable or unsolvable in polynomial time. Even the complexity of a special case of SDP known as the {\it square-root sum} problem is still open.  In many applications, both of LP and SDP, one only seeks to find an $\varepsilon$-approximation solution with sufficiently small $\varepsilon$, hence the  time complexity as a function of $\varepsilon$ is not a major issue.
What then becomes relevant in algorithms for such applications is the interplay between  the dimensions, $n$ and $m$, and the desired tolerance. In particular, in this sense  an FPTAS complexity as a function of $m,n, 1/\varepsilon$ could outperform a polynomial time complexity in terms of $m,n, \ln 1/\varepsilon$  when $\varepsilon$ is in a reasonable range.

The algorithm of  Frank-Wolfe \cite{Frank} is a well-known method for the minimization of a convex function over a compact convex set.
Even when Newton's method is applicable to a convex program, the Frank-Wolfe method compromises on the complexity of each iteration by minimizing a linear approximation to the objective function. It avoids the task of solving a system of linear equations arising in a quadratic approximation to the underlying objective function. Its iteration complexity is basically $O(1/\varepsilon^2)$.  In particular, given a CHM, Frank-Wolfe computes an approximate solution to the nearest point of  $p_\circ$ in $conv(S)$. When $p_\circ$ is not in $conv(S)$, the problem is known as {\it polytope distance}. A more general version of the problem is computing the distance between two convex hulls which arises in the {\it Support Vector Machine} (SVM) applications.  For relevant results on Frank-Wolfe, see   Gilbert \cite{Gilbert}, Clarkson \cite{clark2008} and  G{\"a}rtner and Jaggi \cite{Gartner}, and Jaggi \cite{Jaggi} for general convex optimization over a compact set via Frank-Wolfe.

The {\it Triangle Algorithm} (TA), introduced in \cite{kalchar}, is a geometrically inspired algorithm designed to solve CHM.  TA has similarities to Frank-Wolfe but it offers much more flexibility. It is endowed with dualities and offers alternative complexity bounds. Indeed in several experimentations TA has outperformed Frank-Wolfe.  When $p_\circ \in conv(S)$, in each iteration TA produces a {\it pivot}, guaranteed to exist by a {\it distance duality}, a point in $S$ that allows reducing the gap between the current iterate and $p_\circ$.
Then in $O(1/\varepsilon^2)$ iteration it computes  $p_\varepsilon \in conv(S)$ so that $\Vert p_\circ - p_\varepsilon \Vert \leq \varepsilon R$, where $R= \max \{\Vert p_\circ - v_i \Vert: v_i \in S \}$.  An alternative complexity for TA can be stated if $p_\circ$ is contained in a ball of radius $\rho$, contained in the relative interior of $conv(S)$. Specifically,  the number of iterations to compute an $\varepsilon$-approximate solution $p_\varepsilon$ is $O((R^2/\rho^2) \ln (1/\varepsilon))$.  In \cite{kalsphere} it is shown that in  the {\it Spherical}-CHM, the special case of CHM where $p_\circ =0$ and $\forall v \in S$, $\Vert v \Vert=1$,  assuming a certain checkable condition is satisfied in each iteration, the overall number of iterations is  $O(1/\varepsilon)$. Indeed in numerous  experimental results TA seems to run quite well. When $p_\circ$ is not in $conv(S)$, TA eventually computes a {\it witness}, a point $p'$ in $conv(S)$, where the  orthogonal bisecting hyperplane to the line segment $p_\circ p'$ separates $p_\circ$ from $conv(S)$. In fact $p'$ gives an estimate of the distance from $p_\circ$ to $conv(S)$ to within a factor of two. The complexity of computing a witness is dependent on the distance from $p_\circ$ to $conv(S)$, say $\delta_*$ and thus it follows that the number of iterations is $O(R^2/\delta_*^2)$.  When $\delta_*$ is bounded away from zero, the number of iterations to find a witness is very few.  This is an important feature of TA and has proved to be very useful algorithmically and in several applications. The algorithm {\it AVTA} (All-Vertex Triangle Algorithm) proposed in \cite{AKZ}, repeatedly makes use of TA in order to efficiently  compute the set of all vertices of $conv(S)$, or an approximate subset of the vertices whose convex hull approximates $conv(S)$ to any prescribed accuracy.  Aside from theoretical significance, practical utility of TA is supported by large-scale computations in realistic applications. For applications in machine learning, computational geometry and LP see \cite{AKZ}, where a substantial amount of computing is also provided.

A generalization of the Triangle Algorithm described in \cite{kalsep} tests if a  given pair of arbitrary compact convex sets $C, C' \subset \mathbb{R}^m$ intersect or are separable. Specifically, it either computes an approximate point of intersection, or a separating hyperplane, or if desired an approximation to the optimal pair of supporting hyperplanes, or approximation to the distance between the sets.  In particular, the {\it hard margin} SVM is a very special case where the convex sets are convex hull of finite sets, see \cite{Burges, Gupta}.  In this more general version of TA the complexity of each iteration depends on the nature and description of the underlying sets.  In the worst-case it requires  solving an LP over one or both convex set. However, often times this LP only needs to be solved suboptimally.  In this article we shall refer
to {\it General-CHM} as the problem of testing if a given point $p_\circ \in \mathbb{R}^m$ lies in a given arbitrary compact convex subset $C$ of $\mathbb{R}^m$.

In this article we first reduce a specialized  SDP feasibility problem,  called {\it spectrahull membership} (SHM) to a General-CHM and then develop a version of the Triangle Algorithm to solve it. SHM is an analogue of  CHM for SDP. Just as CHM is closely related to LP feasibility,  SHM is closely related to SDP feasibility and hence SDP optimization.  A general SDP can be viewed as an LP,  where the underlying nonnegativity cone is replaced with the cone of symmetric positive semidefinite matrices, see e.g.  Alizadeh \cite{Alizadeh}, Nesterov and Nemirovskii \cite{NN}, Vandenberghe and Boyd \cite{van96}.  Both LP and SDP are special cases of {\it self-concordant} optimization problems, developed by Nesterov and Nemirovskii \cite{NN} and can be approximated to within $\varepsilon$ tolerance in polynomial time complexity in terms of the dimensions and  $\ln 1/\varepsilon$. The main work in each iteration is solving a Newton system that arises in the course of quadratic approximation to a potential function or a parameterized potential function via a path-following approach. Many interior-point algorithms for LP have been extended to SDP without the use of self-concordance theory. As in LP relaxations, SDP relaxations find applications in combinatorial optimization. In particular, Goemans and Williamson \cite{GW} showed that an SDP relaxation of the {\it MAX CUT} problem produces a good approximation.  Other applications of SDP for combinatorial problems are described by  Lov\'asz \cite{Lovasz}. SDP relaxations are also considered  in non-convex {\it Quadratically Constraint Quadratic Programs}, see e.g  Nesterov et al. \cite{NW} and Luo et al. \cite{Luo}.  As in the MAX CUT problem, an optimal solution of the SDP relaxation has to be rounded into a feasible solution. This rounding however is not necessarily as in the MAX CUT problem because it is dependent on the constraints of the underlying problem.  The significance of solving SDP relaxation thus lies in whether or not it is possible to round its optimal solution into  a feasible solution, as well as the quality of approximate solution it produces.  A survey of SDP relaxations for {\it Binary Quadratic Programming}  with guaranteed approximation is given in \cite{Luo}. The complexity of solving the SDP relaxation via interior-point algorithms, e.g. as stated in Helmberg et al. \cite{HRVW}, is $O(n^{4.5} \ln 1/\varepsilon )$ operations, and in some special cases $O(n^{3.5} \ln 1/\varepsilon)$. However, in a general SDP the number of constraints can be as large as $O(n^2)$ so that the over all complexity becomes
 $O(n^{6.5} \ln 1/\varepsilon )$, see Nesterov \cite{Nesterov}.

One may ask, why not just solve SHM as an SDP optimization?  There are many ways to answer this question and to justify theoretical and practical significance of the study of SHM.  On the one hand the case of homogeneous CHM is realistic and intrinsic to optimization, arising for example as a dual to strict LP feasibility. The study of homogeneous CHM has for example given rise to the {\it diagonal matrix scaling dualities} as well as a simple path-following algorithm for solving CHM or producing a corresponding diagonal scaling, \cite{KK}.  In \cite{kalsdp} it is also shown that analogous dualities and a corresponding algorithm is possible for a homogeneous case of SDP.  On the other hand, CHM is a significant special case of LP feasibility the study of which has given rise to the distance dualities and the Triangle Algorithm which in many applications is capable of producing very good solutions efficiently. When the dimensions in  CHM are very large even performing one Newton iteration in  interior point methods may become prohibitive.  The Triangle Algorithm has performed very well in such instances. Analogously,  the study of SHM together with extension of the Triangle Algorithm will provide an alternative algorithm for SDP feasibility and optimization.  It allows replacing expensive Newton iterations arising in interior point algorithms with efficient iterations, but possibly at the cost of performing more iterations.  A single Newton step in solving an SDP with $m$ constraints with $n \times n$ matrices could take as much as $O(n^2(m+n)m)$ arithmetic operations, see Nesterov \cite{Nesterov}. Here $m$ can be as large as $n(n+1)$ so that the operations of a single iteration can be as large as $O(n^6)$.  The complexity of each iteration in the proposed TA for SHM is $O(mn^2)$ plus that of obtaining a pivot, computable by estimating the least eigenvalue of a symmetric $n \times n$ matrix arising in that iteration. This estimation can be achieved via the {\it power method} and does not need to be carried out to optimality.  However, by proving a semidefinite version of Carath\'eodory theorem and the existence of rank-one pivots, it is possible to implement TA as if it is solving a CHM, but occasionally calling the power method.  In summary, the proposed  Triangle Algorithm is a simple and practical algorithm for SHM with novel dualities, allowing interplay with the standard CHM.  Furthermore, based on existing computational and theoretical experiences with  TA for  CHM, we are led to believe for SHM too  the proposed version of TA  will result in  theoretical and practical alternatives to the existing SDP algorithms. It can also be extended to an SDP analogue of SVM for testing the separation of two SHMs.

The article is organized as follows:  In Section \ref{sec2}, we summarize the relevant properties of the Triangle Algorithm from \cite{kalchar, kalsep} for CHM and the more general case of General-CHM, giving its description,  dualities and complexities.  In Section \ref{sec3}, we define spectrahull, then the spectrahull membership problem (SHM) and give its relation to SDP feasibility.  In Section \ref{sec4}, we prove {\it distance dualities} for SHM and state a semidefinite version of Carath\'eodory theorem.
In Section \ref{sec5},  we give  characterization of a {\it pivot} (a point used to reduce the current gap) for SHM and state several corollaries of this characterization.  In Section \ref{sec6}, we describe a version of TA for SHM. Then in \ref{subsec6.1} we describe a strategy for computing a pivot via the power method. In \ref{subsec6.2} we describe how the Triangle Algorithm for  SHM could interact with an underlying CHM. In \ref{subsec6.3} we give a complexity analysis for solving the SDP relaxation of MAX CUT via the TA. In Section \ref{sec7}, we describe an analogue of SVM for SHM and state the complexity of solving it via the TA.  We end with concluding remarks on the potentials of the proposed algorithm,  its applications and possible extensions.

\section{Summary of the Triangle Algorithm and its Properties} \label{sec2}

The {\it Triangle Algorithm} (TA) introduced in \cite{kalchar} solves the {\it convex hull membership} problem:

\begin{definition}  \label{defn0} {\rm (CHM)}
Given a subset $S=\{v_1, \dots, v_n\} \subset \mathbb{R}^m$, a distinguished point $p_\circ \in \mathbb{R} ^m$, and $\varepsilon \in (0,1)$, solving CHM means either computing an $\varepsilon$-approximate solution, i.e. $p_\varepsilon \in conv(S)$ so that
\begin{equation}
\Vert p_\circ - p_\varepsilon \Vert  \leq \varepsilon R,  \quad R = \max \{\Vert v_i - p_\circ \Vert  : v_i \in S\},
\end{equation}
or a hyperplane that separates $p_\circ$ from $conv(S)$.
\end{definition}

A more general case of CHM and a corresponding Triangle Algorithm were developed in \cite{kalsep}. In particular, it solves the following problem.

\begin{definition}  \label{defn1} {\rm (General-CHM)} Given  an arbitrary compact convex subset $C$ of $\mathbb{R}^m$ (described either explicitly or implicitly),
a distinguished point $p_\circ \in \mathbb{R} ^m$, and $\varepsilon \in (0,1)$, either compute an $\varepsilon$-approximate solution, i.e. $p_\varepsilon \in C$ so that
\begin{equation}
\Vert p_\circ - p_\varepsilon \Vert  \leq \varepsilon R,  \quad R = \max \{\Vert x - p_\circ \Vert  : x \in C\},
\end{equation}
or a hyperplane that separates $p_\circ$ from $C$.
\end{definition}

\begin{remark} By the Krein-Milman Theorem,  $C$  is the convex hull of its extreme points. Thus General-CHM is indeed a general version of CHM albeit the set of extreme points may be infinite and only known implicitly.
\end{remark}

\begin{definition} \label{defn2}
Given  $p' \in C$, $v \in C$  is called a $p_\circ$-{\it pivot} (or  pivot at $p'$, or simply pivot) if
\begin{equation} \label{pivot}
\Vert p' - v \Vert \geq \Vert p_\circ -  v \Vert.
\end{equation}
Equivalently,
\begin{equation} \label{eq2}
(p' - p_\circ)^Tv \leq  \frac{1}{2} (\Vert p' \Vert^2 - \Vert p_\circ \Vert^2 ).
\end{equation}
\end{definition}

\begin{remark}  \label{rem1} From (\ref{eq2}) we have, given $p' \in C$,  $v \in C$ is a pivot at $p'$ if and only if
\begin{equation}  \label{eeq1}
\min\{(p' - p_\circ )^Tx : x \in C\} \leq (p' - p_\circ )^Tv \leq \frac{1}{2} (\Vert p' \Vert^2 -  \Vert p_\circ \Vert^2) .
\end{equation}
It follows that, in the worst-case, testing for the existence of a pivot at $p'$  requires minimizing a linear function over $C$.  However, in many cases this optimization is not necessary, or not needed to be carried out to optimality.
\end{remark}

\begin{definition} \label{witness} We say $p' \in C$ is a
$p_\circ$-{\it witness} (or simply a {\it witness})  if  the orthogonal bisecting hyperplane to the line segment $p_\circ p'$ separates $p_\circ$ from $C$.  Equivalently,
\begin{equation} \label{eq3}
\Vert p' - x \Vert  < \Vert p_\circ - x \Vert, \quad \forall x \in C.
\end{equation}
\end{definition}

Geometrically, the {\it Voronoi cell} of $p'$ (relative to $p_\circ$) contains $C$ in its entirety.  The separating hyperplane $H$ is given as
\begin{equation}
H=\{x: (p'-p_\circ)^Tx = \frac{1}{2}(\Vert p' \Vert^2 -\Vert p_\circ \Vert^2)\}.
\end{equation}

Given an iterate $p' \in C$ that is neither an $\varepsilon$-approximate solution nor a witness, TA finds a $p_\circ$-pivot $v \in C$. Then on the line segment $p'v$ it computes the nearest point to $p_\circ$. It then replaces $p'$ with the nearest point and repeats the process.  The nearest point can be computed easily:

\begin{prop} \label{prop1} Given an iterate $p' \in C$, suppose $v \in C$ is a pivot. If the nearest point to $p_\circ$ on the line segment $p'v$ is not $v$ itself, it is given as
\begin{equation} \label{eq5}
p'' = (1-\alpha)p' + \alpha v, \quad  \alpha = {(p_\circ -p')^T(v-p')}/{\Vert v - p' \Vert^2}.
\end{equation}
\qed
\end{prop}

The next theorem justifies the correctness of TA.

\begin{thm}  \label{thm1} {{\rm (Distance Duality)}}
$p_\circ \in C$ if and only if for each  $p' \in C$ there exists a pivot $v \in C$. Equivalently,  $p_\circ \not \in C$ if and only if there exists a witness $p' \in C$.  \qed
\end{thm}

The first complexity statement for TA is given in the following theorem from \cite{kalsep}.

\begin{thm}   \label{thm2}  {\rm (Complexity Bounds)} Given $\varepsilon \in (0,1)$, in $O(1/\varepsilon^2)$ iterations TA either computes
$p_\varepsilon \in C$ with $\Vert p_\circ - p_\varepsilon \Vert \leq R \varepsilon$, or a witness. In particular, if  $p_\circ \not \in C$ and $\delta_*= \min\{\Vert x-p_\circ \Vert: x \in C\}$, the number of itaretions to compute a witness is
$O(R^2/\delta^2_*)$.  Furthermore, given  any witness $p' \in C$, we have
\begin{equation}
\delta_* \leq \Vert p' -  p_\circ \Vert \leq 2 \delta_*.
\end{equation}
\qed
\end{thm}

In a more complicated fashion, see \cite{kalsep}, it can be shown that TA can approximate the distance to $p_\circ$ to within any prescribed accuracy $\varepsilon$.

\begin{thm}  \label{thm6} Suppose $p_\circ \not \in C$. Let $\delta_*= \min\{\Vert x-p_\circ \Vert: x \in C\}$. Given $\varepsilon \in (0,1)$, in $O(R^2/\delta^2_* \varepsilon)$ iterations TA compute an approximation $p_\varepsilon \in C$ satisfying
\begin{equation}  \label{eeq2}
\Vert p_\varepsilon -p_\circ \Vert - \delta_* \leq  \Vert p_\varepsilon -p_\circ  \Vert \varepsilon.
\end{equation}
\qed
\end{thm}

A stronger version of a pivot can be defined.

\begin{definition} \label{strictpivot} Given $p' \in C$,  $v \in C$ is a {\it strict} $p_\circ$-pivot at $p'$ (or simply a strict pivot) if $\angle p'p_\circ v \geq \pi/2$. Equivalently, $(p'-p_\circ)^T (v-p_\circ) \leq 0$. That is,
\begin{equation}
(p'-p_\circ)^Tv \leq p'^Tp_\circ - \Vert p_\circ \Vert^2.
\end{equation}
\end{definition}

\begin{thm}  \label{thm3} {\rm (Strict Distance Duality)}  $ p_\circ \in C$ if and only if for each  $p' \in C$, $p' \not = p_\circ$, there exists a strict $p_\circ$-pivot $v \in C$ (here $v$ can coincide with $p_\circ$ itself). \qed
\end{thm}

An alternative complexity bound can be stated.

\begin{thm}  \label{thm4} Suppose a ball of radius $\rho >0$ centered at $p_\circ$ is contained in the relative interior of $C$. If TA uses a strict pivot in each iteration,  $p_\varepsilon \in C$ satisfying $\Vert p_\varepsilon - p_\circ \Vert  \leq \varepsilon$  can be computed in $O\big ((R/\rho)^{2} \ln {1}/{\varepsilon} \big )$ iterations. \qed
\end{thm}

\begin{definition}  A {\it Spherical-CHM} is the case of CHM where each $v \in S$ has unit norm and $p_\circ=0$. We say it has $\varepsilon$-property at $p' \in conv(S)$ with $\Vert p' \Vert > \varepsilon$, if there is a pivot $v$ such that
\begin{equation}
\Vert p' - v \Vert \geq \sqrt{1+ \varepsilon}.
\end{equation}
\end{definition}

As an example if the ball of radius $\sqrt{\varepsilon}$ is contained in $conv(S)$, then  Spherical-CHM has $\varepsilon$-property everywhere.   When this property is satisfied at each iteration we have the following improved complexity.

\begin{thm}  Consider a Spherical-CHM. If every iterate $p'$ in TA with $\Vert p' \Vert > \varepsilon$ that is not a witness has $\varepsilon$-property, then in $O(1/\varepsilon)$ iterations, TA either computes  a witness, or $p_\varepsilon \in conv(S)$ such that $\Vert p_\varepsilon \Vert \leq \varepsilon$. \qed
\end{thm}

The overall number of iterations in TA is independent of the nature of $C$.  However, the complexity of each iteration is dependent on $C$ and its description.
The simplest case is when $C=conv(S)$, $S$ a finite set of points. The following theorem, proved in \cite{AKZ}, was stated for the case where $C= conv(S)$, however it can be stated for the general case of $C$. We thus state the theorem for this general case as we will use it in this article.

\begin{thm}  \label{thmCHMX}  Consider the General-CHM. Let $\widehat S= \{ \widehat v_1, \dots, \widehat v_N\}$ be a subset  of $C$. Given $p_\circ \in \mathbb{R}^m$, consider testing if $p_\circ \in C$.
Given $\varepsilon \in (0, 1)$, the complexity of testing if there exists an $\varepsilon$-approximate solution in $conv(\widehat S)$ is
\begin{equation} \label{eqAA}
O \bigg (m N^2+ \frac{N}{\varepsilon^2} \bigg).
\end{equation}

In particular, suppose in testing if $p_\circ \in C$, the Triangle Algorithm computes an $\varepsilon$-approximate solution $p_\varepsilon$ by examining only the elements of a subset $\widehat S= \{ \widehat v_1, \dots, \widehat v_N\}$ of $C$.  Then the complexity of testing if  there exists an $\varepsilon$-approximate solution  $p_\varepsilon \in C$ is  as stated in (\ref{eqAA}). \qed
\end{thm}

We will use the above results in the article to describe a version of TA for SHM and to derive its complexity.  The proposed Spectrahull Triangle Algorithm can essentially solve SDP feasibility having no recession direction and also  SDP optimization over a bounded feasible set. However, in this article we only focus on SHM.

\section{Spectrahull and Spectrahull Membership Problem} \label{sec3}
Let $\mathbb{S}^n$ denote the set of $n \times n$ real symmetric matrices. As usual the notation $X \succeq 0$ means $X$ lies in $\mathbb{S}_+^n$, the cone of positive semidefinite matrices in $\mathbb{S}^n$. $\mathbb{S}^n$ is a Hilbert space, where its inner product, also refereed as Frobenious inner product, is denoted by any of the following equivalent notations
\begin{equation}  \label{eeq3}
\langle X, Y \rangle_F=Tr(XY)=X \bullet Y = \sum_{i=1}^n \sum_{j=1}^n x_{ij} y_{ij}.
\end{equation}
The corresponding induced norm is\emph{}
\begin{equation}  \label{eeq4}
\Vert X \Vert_F= Tr(X^2)^{1/2}= (X \bullet X)^{1/2}.
\end{equation}

The term {\it spectraplex}, defined in \cite{Parrilo}, is associated with the set

\begin{equation}  \label{eeq5}
\mathbf{\Delta}_n=\{X \in \mathbb{S}^n:  Tr(X)=1, \quad  X \succeq 0  \}.
\end{equation}

It is  an example of {\it spectrahedron}, see \cite{Ramana}. It is an analogue of the unit simplex in $\mathbb{R}^n$.  It is known that the extreme points of $\mathbf{\Delta}_n$
are rank-one matrices of the form $vv^T$, where $v \in \mathbb{R}^n$ is of unit norm (see Lemma \ref{lem1} and Corollary \ref{spectraplex}).

\begin{definition}  \label{SDP-Hull} Given a subset $\mathbf{S}=\{A_1, \dots A_m\}$ of $\mathbb{S}^n$, we define its {\it spectrahull}, denoted by $SH(\mathbf{S})$,  as
\begin{equation}  \label{eeq6}
\bigg \{p(X) \equiv (A_1 \bullet X, \dots, A_m \bullet X)^T: X \in \mathbf{\Delta}_n \bigg \}.
\end{equation}
\end{definition}

\begin{remark}  \label{rem2} $SH(\mathbf{S})$ is a subset of $\mathbb{R}^m$, defined by the linear map $p: \mathbf{\Delta}_n  \rightarrow \mathbb{R}^m$.  Suppose for each $i=1, \dots, m$, $A_i$ is a diagonal matrix, say ${\rm diag}(a^{(i)})$, $a^{(i)} \in \mathbb{R}^n$.  For $j=1, \dots, n$
setting $v_j= (a^{(1)}_j, \dots, a^{(m)}_j)^T \in \mathbb{R}^m$, it follows that if  $S= \{v_1, \dots, v_n\}$, then $SH(\mathbf{S})= conv(S)$,
the standard convex hull of $S$. This justifies the name spectrahull.
\end{remark}

\begin{definition}  \label{SHMP} Given $\mathbf{S}=\{A_1, \dots A_m\} \subset   \mathbb{S}^n$, and $b \in \mathbb{R}^m$, let
{\it spectrahull membership} (SHM) stand for testing if $b \in SH(\mathbf{S})$.
\end{definition}

\begin{remark} \label{rem3} When $A_i$'s are diagonal matrices as in  Remark \ref{rem2},  SHM is identical with the CHM that tests if $b \in conv(S)$.
That is,  $b \in SH(\mathbf{S})$ if and only if  $b \in conv(S)$.
\end{remark}

Given $\mathbf{S}=\{A_1, \dots A_m\} \subset   \mathbb{S}^n$, and $b \in \mathbb{R}^m$, $b \not = 0$, {\it SDP feasibility} is testing the feasibility of:
\begin{equation} \label{eqp}
\mathbf{P}=\{X \in \mathbb{S}^n:  A_i \bullet X = b_i, \quad i=1, \dots, m,  \quad X \succeq 0\}.
\end{equation}

Just as homogeneous CHM is a fundamental problem in LP,  homogenous SHM, the case where  $b= 0$, is a fundamental problem in SDP. As shown next, the bounded case of SDP feasibility is reducible to a homogeneous SHM.

\begin{prop} \label{prop2} Suppose $\mathbf{P}$  given in (\ref{eqp}) has no recession direction, i.e.
\begin{equation}
Rec(\mathbf{P})= \{D \in \mathbb{S}^n:  A_i \bullet D =0, \quad  i=1, \dots, m, \quad D \succeq 0, \quad D \not =0 \} = \emptyset.
\end{equation}
Let $\mathbf{S}'=\{A'_1, \dots A'_m\} \subset   \mathbb{S}^{n+1}$, where
\begin{equation}  \label{eeq8}
A'_i={\rm diag}(A_i, -b_i)=
\begin{pmatrix}
A_i  &0 \\
0&-b_i\\
\end{pmatrix}.
\end{equation}
Then $\mathbf{P}$ is feasible if and only if
\begin{equation}  \label{eeq9}
0\in SH(\mathbf{S}') =
\bigg \{p(X')=(A_1' \bullet X', \dots, A_m' \bullet X')^T: \quad  X' \in \mathbf{\Delta}_{n+1} \bigg \}.
\end{equation}
\end{prop}
\begin{proof} Assume $X \in \mathbf{P}$. Letting $X'={\rm diag}(X,1)$, it is easy to see $p(X')=0$. Clearly, $X' \succeq 0$, $Tr(X') >0$.  Thus letting  $X''=X'/Tr(X')$ proves $0 \in SH(\mathbf{S}')$.
Conversely, suppose for some $X'  \in \mathbf{\Delta}_{n+1}$, $p(X')=0$.  We can thus write
\begin{equation}  \label{eeq10}
X'=
\begin{pmatrix}
X  &v \\
v^T&\alpha\\
\end{pmatrix}, \quad X \in \mathbb{S}^n,  \quad v \in \mathbb{R}^n, \quad  \alpha \geq 0, \quad Tr(X)+ \alpha =1.
\end{equation}
From the structure of $A_i'$ we have,
\begin{equation}  \label{eeq11}
A_i' \bullet X'= A_i \bullet X - \alpha b_i=0,  \quad i=1, \dots, m.
\end{equation}
But $\alpha \not =0$, otherwise $\mathbf{P}$ has a recession direction. Dividing equations in (\ref{eeq11}) by $\alpha$, we get  $X/\alpha \in \mathbf{P}$.
\end{proof}

\section{Distance Duality and Carath\'eodory Theorems for SHM} \label{sec4}

\begin{prop} \label{prop3} Given $\mathbf{S}= \{A_1, \dots, A_m\} \subset \mathbb{S}^n$, $SH(\mathbf{S})$ is a compact convex subset of $\mathbb{R}^m$. In particular, for each $p_\circ \in SH(\mathbf{S})$ there exists $X_\circ \in \mathbf{\Delta}_n$ such that $p(X_\circ)=p_\circ$.
Furthermore, if $p_\circ$ is an extreme point of $SH(\mathbf{S})$,
there exists an extreme point $X_\circ$ of $\mathbf{\Delta}_n$ such that $p(X_\circ)=p_\circ$.
\end{prop}
\begin{proof} Given $X, Y \in  \mathbf{\Delta}_n$, for each $\alpha \in [0,1]$, $\alpha X+(1-\alpha)Y \in \mathbf{\Delta}_n$. Thus $SH(\mathbf{S})$ is convex.   For each $X \in \mathbf{\Delta}_n$, $\Vert X \Vert_F \leq 1$.  Thus

\begin{equation}  \label{eeq12}
|A_i \bullet X| \leq \Vert A_i \Vert_F \Vert X \Vert_F \leq \Vert A_i \Vert_F, \quad  \forall i=1, \dots, m.
\end{equation}
Hence $SH(\mathbf{S})$ is  bounded.  Let $\{X_k \in \mathbf{\Delta}_n: k \geq 1\}$ be a sequence with
$p(X_k)$ convergent to $u \in \mathbb{R}^m$.  Since $\mathbf{\Delta}_n$ is compact, the sequence  has an accumulation point, $X_* \in \mathbf{\Delta}_n$. Then by continuity of the map $p(X)$, we have $u=p(X_*)$.
Thus $SH(\mathbf{S})$ is closed.   Hence $SH(\mathbf{S})$ is compact.

Suppose $p_\circ$ is an extreme point of $SH(\mathbf{S})$ but whenever $p_\circ=p(X_\circ)$,
$X_\circ \in \mathbf{\Delta}_n$, $X_\circ$ is not an extreme point. Then $X_\circ$ is convex combination of a finite number, $t$, of extreme  points of $\mathbf{\Delta}_n$,
$X_\circ= \sum_{i=1}^t \alpha_i X_i$, $\sum_{i=1}^t \alpha_i=1$, $\alpha_i >0$, each $X_i$ an extreme point of $\mathbf{\Delta}_n$. By assumption then $p(X_i) \not = p_\circ$ for all $i$.  However, $p(X_\circ)= \sum_{i=1}^t\alpha_ip(X_i)$,  contradicting that $p_\circ$ is a an extreme point.
\end{proof}

\begin{remark} The spectrahull  $SH(\mathbf{S})$ is the image of $\mathbf{\Delta}_n$ under the mapping $p(X)= (A_1 \bullet X, \dots, A_m \bullet X)^T$.  According to the above theorem the inverse image of a vertex includes a vertex of $\mathbf{\Delta}_n$.  The simplest case of $SH(\mathbf{S})$ is when $m=1$. In this case $SH(\mathbf{S})$ is merely an interval and can be shown to have the extreme eigenvalues of $A_1$ as its endpoints, see Lemma \ref {lem1}.
\end{remark}

\begin{thm} {\rm (Distance Dualities)} The following two conditions are equivalent:

(1)  $b \in SH(\mathbf{S})$.

(2) For each $X'  \in \mathbf{\Delta}_n$, there exists $V \in \mathbf{\Delta}_n$ such that
\begin{equation} \label{eeq12half}
\Vert p(X')- p(V)\Vert \geq \Vert b-p(V) \Vert.
\end{equation}
Equivalently, the following two conditions are equivalent:

(1') $b \not \in SH(\mathbf{S})$

(2') There exists $W \in \mathbf{\Delta}_n$  such that
\begin{equation}  \label{eeq13}
\Vert p(W)- p(X) \Vert < \Vert b - p(X) \Vert, \quad \forall X \in \mathbf{\Delta}_n.
\end{equation}
\end{thm}
\begin{proof} Assume (1) is true. Given $X' \in \mathbf{\Delta}_n$, let $p'=p(X') \in SH(\mathbf{S})$. Then from  the distance duality theorem for General-CHM, Theorem \ref{thm1}, there exists a pivot $v \in SH(\mathbf{S})$ such that
\begin{equation}  \label{eeq14}
\Vert p' - v \Vert \geq \Vert b-v \Vert.
\end{equation}
From Proposition \ref{prop3}, $v=p(V)$ for some $V \in \mathbf{\Delta}_n$.  Thus (1) implies (2). Assume (2) holds. We prove algorithmically that (2) implies (1), using the Triangle Algorithm. Start with an arbitrary $X'$ in $\mathbf{\Delta}_n$. Then there exists $V \in \mathbf{\Delta}_n$  such that $\Vert p(X')- p(V)\Vert \geq \Vert b-p(V)\Vert$.  Let
\begin{equation}  \label{eeq15}
\alpha_* = {\rm argmin} \{ \Vert (1-\alpha) p(X') + \alpha p(V) \Vert: \alpha \in [0,1]\}.
\end{equation}
By linearity of $p(X)$,
\begin{equation}  \label{eeq16}
\Vert (1-\alpha_*) p(X') + \alpha_* p(V) \Vert =  \Vert p\big ((1-\alpha_*) X +  \alpha_* V \big) \Vert.
\end{equation}
Let
\begin{equation}  \label{eeq17}
X''= (1-\alpha_*) X' + \alpha_* V.
\end{equation}
Clearly $X'' \in \mathbf{\Delta}_n$. In this fashion, given $X' \in \mathbf{\Delta}_n$, we replace it with $X'' \in \mathbf{\Delta}_n$ and repeat the process. If this generates the sequence $X_k \in \mathbf{\Delta}_n$,  from  the convergence  analysis of the  Triangle Algorithm $\Vert b - p(X_k) \Vert$ converges to zero.  Then by the compactness of $SH(\mathbf{S})$ and $\mathbf{\Delta}_n$ it follows that there is an accumulation point of $X_k$'s, say $X_* \in \mathbf{\Delta}_n$ for which $p(X_*)=b$. Hence (1) holds.  The equivalence of (1'), (2') can be proved analogously.
\end{proof}

\begin{remark}   \label{rem4} In particular, SHM is equivalent to a General-CHM (see Definition \ref{defn1}), where  the vertices of $C$ are known implicitly.  More specifically, the computation of a pivot requires solving a special SDP problem, one that amounts to estimating the least  eigenvalue of a symmetric matrix. This task can be accomplished via the iterations of the power method. This will be treated in detail in the next section.
\end{remark}

In the remaining of this section we give a version of Carath\'eodory Theorem for SHM.  However, we need to
give an alternate form of the standard version.  According to the standard case of the theorem, if $p_\circ \in conv(\{v_1, \dots, v_n \}) \subset \mathbb{R}^m$, then $p_\circ$ lies in the convex hull of at most $m+1$ of the $v_i$'s.  Alternatively, the theorem can be restated as follows:  If $p_\circ \in conv(\{v_1, \dots, v_n \})$, there exists $x$ in the unit simplex $S_n= conv(\{e_1, \dots, e_n\})$
that is a convex combination of at most $m+1$ vertices of $S_n$. Furthermore, if $x = \sum_{i=1}^n x_ie_i$ is this convex combination, then $p_\circ= \sum_{i=1}^n x_iv_i$. In particular, $p_\circ$ is a convex combination of at most $m+1$ $v_i$'s.

\begin{thm} {\rm (Carath\'eodory Theorem for SHM)} \label{Carat} Given $\mathbf{S}=\{A_1, \dots A_m\} \subset   \mathbb{S}^n$ and $b \in \mathbb{R}^m$,
if $b \in SH(\mathbf{S})$, there  exists $X \in \mathbf{\Delta}_n$ with $p(X)=(A_1 \bullet X, \dots, A_m \bullet X)^T=b$, where $X$ can be written as a convex combination of at most $m+1$ extreme points of $\mathbf{\Delta}_n$.  More precisely, there exists $X_1, \dots, X_t$ extreme points of $\mathbf{\Delta}_n$, $t \leq m+1$,
such that $X \in conv(\{X_1, \dots, X_t\})$, i.e.
\begin{equation}  \label{eeq19}
X= \sum_{i=1}^t \alpha_i X_i, \quad \sum_{i=1}^t \alpha_i =1,  \quad \alpha_i \geq 0.
\end{equation}
\end{thm}
\begin{proof} By Proposition \ref{prop3}, $SH(\mathbf{S})$ is a compact convex subset of $\mathbb{R}^m$.  It is thus the convex hull of its extreme points (Krein-Milman Theorem).  Thus $b \in SH(\mathbf{S})$ implies,  by the standard Carath\'eodory theorem, $b$
can be written as a convex combination of $t \leq m+1$ vertices of $SH(\mathbf{S})$. Hence
\begin{equation}  \label{eeq20}
b = \sum_{i=1}^t \alpha_i p_i, \quad \sum_{i=1}^t \alpha_i =1, \quad \alpha_i \geq 0,
\end{equation}
where  for each $i=1, \dots, t$, $p_i$  is an extreme point of $SH(\mathbf{S})$.
By Proposition \ref{prop3}, for each $i=1, \dots, t$,  $p_i=p(X_i)$ for some extreme point  $X_i \in \mathbf{\Delta}_n$. We thus define $X \in \mathbf{\Delta}_n$ as
\begin{equation}  \label{eeq20x}
X = \sum_{i=1}^t \alpha_i X_i,
\end{equation}
where $\alpha_i$'s are as in (\ref{eeq20}). By linearity $p(X)=b$.
Hence the proof.
\end{proof}

In the next section we describe stronger versions of the theorem.
\section{Characterization of a Pivot in SHM and its Implications} \label{sec5}

\begin{definition} \label{defpivot}
We say  $V \in \mathbf{\Delta}_n$ is a $b$-pivot at $X'\in \mathbf{\Delta}_n$ if
$p(V)$ is $b$-{\it pivot} at $p(X')$. We say $W \in \mathbf{\Delta}_n$  is a $b$-{\it witness} at $X'\in \mathbf{\Delta}_n$ if $p(W)$ is a $b$-witness at $p(X')$.
\end{definition}

The main step in the development of the Triangle Algorithm for SHM is to compute for a given $X'\in \mathbf{\Delta}_n$ a $b$-pivot $V \in \mathbf{\Delta}_n$, or in the absence of a pivot, a $b$-witness $W \in \mathbf{\Delta}_n$. A pivot allows getting closer to $b$ while a witness induces a separation. We need to analyze the computation of a pivot.

From Definition \ref{defpivot} and Remark \ref{rem1} we have,
  $V \in \mathbf{\Delta}_n$ is a $b$-pivot at $X'$ if and only if
\begin{equation}  \label{eeq22}
\min \big \{(p(X') - b)^T p(X):  X \in \mathbf{\Delta}_n  \big \} \leq
\big (p(X') - b \big )^T p(V) \leq
\frac{1}{2} (\Vert p(X') \Vert^2 - \Vert b \Vert^2).
\end{equation}
Note that
\begin{equation}  \label{eeq23}
\big (p(X') -b \big )^Tp(V)= \sum_{i=1}^m \big (A_i \bullet X' - b_i \big ) (A_i \bullet V)= \big (\sum_{i=1}^m (A_i \bullet X' - b_i) A_i \big ) \bullet V.
\end{equation}
Thus we may state the following characterization of a pivot.

\begin{prop} \label{prop5} Given $X' \in \mathbf{\Delta}_n$, set $A=  \sum_{i=1}^m (A_i \bullet X' - b_i) A_i$.  Then $A \in \mathbb{S}^n$ and $V \in \mathbf{\Delta}_n$ is a pivot at $X'$ if
and only if
\begin{equation}  \label{eeq25}
\min \{A \bullet X:  X \in \mathbf{\Delta}_n\} \leq  A \bullet V \leq \frac{1}{2} \bigg (\Vert p(X') \Vert^2 - \Vert b \Vert^2 \bigg ).
\end{equation}
\qed
\end{prop}

The following characterizes the optimization in (\ref{eeq25}).

\begin{lemma}  \label{lem1} Given arbitrary $A \in \mathbb{S}^n$,
let $\lambda_{\rm min}(A)$ be the minimum eigenvalue of $A$ and $u_*$  a corresponding eigenvector.  Then
\begin{equation} \label{lem1eq1}
\min \{A \bullet X :  X \in \mathbf{\Delta}_n \}=\lambda_{\rm min}(A)=A \bullet u_*u_*^T.
\end{equation}

\end{lemma}
\begin{proof}
To prove the minimum is $\lambda_{\rm min}(A)$ one can show that the matrix $A$ in the optimization problem can be replaced with the diagonal matrix $\Lambda$ of its eigenvalues. From this it is easy to give the proof. However, we give an alternate proof using a self-contained duality result for this very special case of SDP. Note  $\mathbf{\Delta}_n=\{X \in \mathbb{S}^n:   Tr(X)=I_n \bullet X=1, \quad X \succeq 0 \}$, where $I_n$ denotes the $n \times n$ identity matrix. Consider the optimization:
\begin{equation} \label{eqsdd}
\max \{y :  yI_n + S=A,  \quad S \succeq 0\}.
\end{equation}
Let $X$ be any feasible point of (\ref{lem1eq1})
and $(y,S)$  any feasible point. Then from (\ref{eqsdd}) we get
\begin{equation}
y I_n \bullet X + S \bullet X= A \bullet X.
\end{equation}
But $I_n \bullet X=1$ and since  $S \bullet X \geq 0$,  we get $y \leq A \bullet X$. On the other hand, consider the spectral decomposition $A=U\Lambda U^T$, with $\Lambda={\rm diag}(\lambda)$, diagonal matrix of eigenvalues.
Multiplying the equation $yI_n + X=A$ by $U^T$ on the left and $U$ on the right we get, $yI_n+ D= \Lambda$, where $D=U^TXU$ must necessarily be a PSD diagonal matrix. But this implies $y \leq  \lambda_i$ for all $i=1, \dots n$. On the other hand,  the maximum value of $y$ is $\lambda_{\rm min}(A)$, occurring at $X_*=u_*u_*^T$ which coincides  with $A \bullet X_*$.  Hence the proof.
\end{proof}

From Lemma \ref{lem1}, given  $A \in \mathbb{S}^n$, the minimum of $A \bullet X$ over $\mathbf{\Delta}_n$ occurs at a matrix of rank-one. Also, given $v \in \mathbb{R}^n$ of unit norm, the minimum of $vv^T \bullet X$ over $\mathbf{\Delta}_n$ occurs at $vv^T$.  Thus we have

\begin{cor} \label{spectraplex} The extreme points of $\mathbf{\Delta}_n$ are matrices of the form $vv^T$, where $v \in \mathbb{R}^n$, $ \Vert v \Vert =1$.  \qed
\end{cor}

Given $X \in \mathbf{\Delta}_n$, on the one hand from spectral decomposition  $X=\sum_{i=1}^n \lambda_i u_i^Tu_i$ with $\lambda_i$, $u_i$ its eigenvalue-eigenvectors.  On the other hand, $Tr(X)=1$.  From this,  Corollary \ref{spectraplex} together with Theorem \ref{Carat} we have

\begin{thm} {\rm (Stronger Carath\'eodory Theorem for SHM)} \label{strongCara}Given $\mathbf{S}=\{A_1, \dots A_m\} \subset   \mathbb{S}^n$ and $b \in \mathbb{R}^m$,
if $b \in SH(\mathbf{S})$, there  exists $X \in \mathbf{\Delta}_n$ with $p(X)=(A_1 \bullet X, \dots, A_m \bullet X)^T=b$, where $X$ can be written as a convex combination of at most $ \min \{m+1, n \}$ rank-one matrices of the form  $vv^T$, where $\Vert v \Vert =1$. \qed
\end{thm}

Theorem  \ref{strongCara} supports a strategy in the design of a  Triangle Algorithm for SHM that makes it run more as if it is a Triangle Algorithm for CHM, hence making it less dependent on the use of power method. This will be treated in detail later.

In fact we can state even a stronger Carath\'eodory-type theorem for SHM.  First, we state the following characterization theorem for SDP feasibility.

\begin{thm} \label{Barv} {\rm (Barvinok \cite{Barvinok})} Given $A_i \in \mathbb{S}^n$, $b_i \in \mathbb{R}$,  if there is a solution $X \succeq 0$ such that $A_i \bullet X=b_i$, $i=1, \dots, m$, then there is a solution $X$ whose rank $r$ satisfies
\begin{equation} r \leq \bigg [ \frac{\sqrt{8m +1} -1}{2} \bigg ],
\end{equation}
where $[ \cdot ]$ means the integer part. \qed
\end{thm}

We may state the following stronger version of Theorem \ref{strongCara}:

\begin{thm} \label{strongerCara} If $b \in SH(\mathbf{S})$, there  exists $X \in \mathbf{\Delta}_n$ with $p(X)=(A_1 \bullet X, \dots, A_m \bullet X)^T=b$, where $X$ can be written as a convex combination of at most
\begin{equation}
r \leq \bigg [ \frac{\sqrt{8m +9} -1}{2} \bigg ],
\end{equation}
rank-one matrices of the form  $vv^T$, where $\Vert v \Vert =1$.
\end{thm}
\begin{proof} The bound on the rank is simply from replacing $m$ in Theorem \ref{Barv} with $m+1$. The fact that $X$ is a convex combination of $r$ rank-one matrices follows from two observations: Firstly,  $X$ is  PSD thus, it can be  written as $\sum_{i=1}^n \lambda_i v_iv_i^T$, where $\lambda_i$-$v_i$ are its eigenvalue-eigenvector decomposition. Secondly,  $\sum_{i=1}^n \lambda_i = \sum_{i=1, \lambda_i >0}^n =Tr(X)=1$.
\end{proof}

According to Theorem \ref{strongerCara} if SHM is feasible, a solution  $X \in \mathbf{\Delta}_n$ of rank $O(\sqrt{m})$ can be guaranteed. In the worst-case $m=O(n^2)$ and thus the bound  in Theorem \ref{strongerCara} is not stronger than that of Theorem \ref{strongCara}, however  for a large range of values of $m$, in particular when $m  \leq n$, Theorem \ref{strongerCara} gives a stronger bound. The drawback however is that computing a minimum rank solution to SDP feasibility is NP-Hard. This problem is analogous to the problem of computing a minimum-support solution to a linear system, known to be NP-hard, see \cite{Garey}.  Nevertheless, there are heuristic approaches for computing a low rank solution to an SDP, see Lemon et al. \cite{Lemon} and references therein.

Relevant properties of  pivot-witness for SHM  are summarized in the following.

\begin{thm} \label{thmxx} Consider SHM.  Given $X' \in \mathbf{\Delta}_n$, set
\begin{equation}  \label{eeq24}
A=  \sum_{i=1}^m (A_i \bullet X' - b_i) A_i.
\end{equation}

If there is a $b$-pivot $V \in \mathbf{\Delta}_n$, then  there is a rank-one $b$-pivot, $V = vv^T$,  $\Vert v \Vert=1$ satisfying
\begin{equation}  \label{eeq25x}
\lambda_{\rm min}(A) \leq  A \bullet V \leq \frac{1}{2} \big (\Vert p(X') \Vert^2 - \Vert b \Vert^2 \big ).
\end{equation}
In particular, in order to compute a pivot $V \in \mathbf{\Delta}_n$ at a given iterate $X' \in \mathbf{\Delta}_n$ it suffices to estimate an eigenvector corresponding to the least eigenvalue of $A$.  Furthermore, if there is a $b$-witness, then there is a rank-one $b$-witness, $W=ww^T$, with $\Vert w \Vert=1$.
\end{thm}
\begin{proof} From Lemma \ref{lem1} we only need to take $v$ (or $w$) to be an eigenvector corresponding to the minimum eigenvalue.
\end{proof}

\section{A Triangle Algorithm for SHM} \label{sec6}

To describe the Triangle Algorithm for testing if a given $b \in \mathbb{R}^m$ lies in $SH(\mathbf{S})$, we first compute an upper  bound on the quantity $R$:
\begin{equation} \label{eqR}
R = \max \{\Vert b - p(X) \Vert : X \in \mathbf{\Delta}_n\} \leq
\Vert b \Vert +  \max \{\Vert p(X) \Vert : X \in \mathbf{\Delta}_n \}.
\end{equation}
Although to run the Triangle Algorithm we do not need such a bound (we simply stop if desirable approximation is achieved), we derive one
for the purpose of complexity analysis.  From (\ref{eqR}), to get a rough bound, it suffices to bound the diameter of $SH(\mathbf{S})$.

\begin{prop} \label{prop10} Given $A \in \mathbb{S}^n$, let $\sigma_{\max(A)}$ be its maximum absolute eigenvalue. Then
\begin{equation}  \label{eeq27}
\max \{|A \bullet X|: X \in \mathbf{\Delta}_n \} = \Vert A \Vert_2 =\sigma_{\max(A)} \leq \Vert A \Vert_F.
\end{equation}
In particular,
\begin{equation} \label{eeq27a}
\max \{\Vert p(X) \Vert : X \in \mathbf{\Delta}_n \} \leq \sum_{i=1}^m \Vert A_i \Vert_2
\end{equation}
\end{prop}

\begin{proof}  The equality in (\ref{eeq27}) follows from Lemma \ref{lem1}. The bound in (\ref{eeq27}) is well known as is the fact that Frobenius norm is square-root of the sum of square of its eigenvalues. The bound in (\ref{eeq27a}) is obvious.
\end{proof}

The Triangle Algorithm works as follows. Step 0 chooses a starting point. Step 1, given an iterate  $X' \in \mathbf{\Delta}_n$, it computes $p(X')$.  Step 2 tests if $p(X')$ is an $\varepsilon$-approximate solution, if so, it terminates. Otherwise, Step 3 computes the matrix $A$ whose least eigenvalue is to be estimated in order to test if a pivot exists. If a pivot $V$ exists, Step 4  computes the closest point to $b$ on the line segment joining $p(X')$ and $p(V)$. The corresponding step size $\alpha_*$ is used to define the next iterate replacing $X'$ and returns to Step 1.  If no pivot is found, Step 5 outputs a witness and terminates.

\begin{algorithm}[H]
		\caption{Spectrahull Triangle Algorithm ($\mathbf{S}=\{A_1, \dots, A_m\} \subset \mathbb{S}^n$,
$b \in \mathbb{R}^m$, $\varepsilon \in (0,1)$)}
		\begin{algorithmic}[1]
			\scriptsize
			\STATE	{\bf Step 0.} Set $X'=ee^T/n$, $e=(1, \dots, 1)^T$.
			\STATE 	{\bf Step 1.}  Set $p(X') = (A_1 \bullet X', \dots, A_m \bullet X')^T$.
\STATE {\bf Step 2.} If $\Vert p(X') - b \Vert \leq \varepsilon R$, output $X'$, stop.
			\STATE  {\bf Step 3.} Set $A=  \sum_{i=1}^m (A_i \bullet X' - b_i) A_i$.
\STATE{\bf Step 4.} If there exists   $V \in \mathbf{\Delta}_n$ (pivot) such that
$A \bullet V \leq \frac{1}{2} \big (\Vert p(X') \Vert^2 - \Vert b \Vert^2 \big )$;

 set $\alpha_* = {\rm argmin} \big \{ \Vert (1-\alpha) p(X') + \alpha) p(V) - b\Vert: \alpha \in [0,1] \big \}$; replace $X'$ with   $(1-\alpha_*) X' + \alpha_* V$; goto Step 1.
			\STATE 	{\bf Step 5.}  Output a witness $W$. Stop.	
		\end{algorithmic}
	\end{algorithm}
	
The iteration complexity of the Triangle Algorithm for SHM
can be stated according to Theorem \ref{thm2}. Also, Theorem \ref{thm4} is applicable if in Step 4 the algorithm computes a strict pivot.  The complexity of each iteration
is dominated by the computation of $p(X')$  and $A$ in Steps 1 and 3 in $O(mn^2)$ time,  plus that of computing a pivot or a witness in Step 4 which either defines the next iterate, returning to Step 1,  or terminates the algorithm in Step 5.  We next discuss a strategy for computing a pivot via the power method.

\subsection{A Strategy for Computing a Pivot Via the Power Method} \label{subsec6.1}
Given an iterate $X' \in \mathbf{\Delta}_n$, we wish to test if there exists a $b$-pivot at $X'$. From Theorem \ref{thmxx} it suffices to estimate the least eigenvalue of the matrix $A=  \sum_{i=1}^m (A_i \bullet X' - b_i) A_i$, see (\ref{eeq24}). This in turn gives an estimate to a corresponding eigenvector.
Conditions are well known when the power method can be used to compute the dominant absolute eigenvalue of a matrix, see e.g. \cite{Golub}. This however may be different from the least eigenvalue of $A$.  The iteration of the power method involve repeated  multiplication of $A$ by a vector. We simplify $A$ into a single matrix. Then each iteration takes $O(n^2)$.  The iterations of the power method estimating the dominant eigenvalue and a corresponding eigenvector of $A$, starting at a random  vector of unit norm $v_\circ \in \mathbb{R}^n$ are  defined as follows
\begin{equation}
\lambda_{k}= v_k^TAv_k, \quad  v_{k+1} =  \frac{A v_k}{\Vert Av_k \Vert}, \quad  k \geq 0.
\end{equation}
If eigenvalues of $A$ are all non-positive (i.e. $-A$ is PSD) the dominant absolute value is the least eigenvalue. Thus iterations of the power method would estimate the least eigenvalue. If $A$ has positive and negative eigenvalues we can shift $A$ by adding an appropriate multiple of the identity matrix.  Let us assume we have adjusted the power method  so that in each iteration it produces estimates $\lambda_k$ and $v_k$ to the least eigenvalue and eigenvector of $A$, respectively. To check if we have a pivot at hand at
$X' \in \mathbf{\Delta}_n$, using Proposition \ref{prop5} and Lemma \ref{lem1}, in each iteration of the power method we test if
\begin{equation} \label{power}
\lambda_k = v_k^TAv_k= A v_kv_k^T \leq \frac{1}{2} \big (\Vert p(X') \Vert^2 - \Vert b \Vert^2 \big).
\end{equation}
That is, in each iteration of the power method the matrix $V_k=v_kv_k^T$ is a candidate to be a pivot.  If the inequality in (\ref{power}) is satisfied we have a pivot and proceed with Step 4 of the Triangle Algorithm. In the worst case we will need to iterate until we have a sufficiently good estimate to the minimum eigenvalue of $A$. This is the case if no pivot exists in which case  a sufficiently good estimate to the minimum eigenvalue gives a witness.  However, we only need to compute a witness once.  We now discuss a strategy that attempts to ensure  the iterates of the power method will estimate the least eigenvalue of $A$.

It is easy to compute bounds on the modulus of the eigenvalues, for instance via the Gershgorin circle theorem. Let us assume we have such a bound $\sigma$. We  apply the power method to the PSD matrix $M=\sigma I - A$.  The largest eigenvalue of $M$ gives the least eigenvalue of $A$.  The advantage in applying the power method to a PSD matrix is that the following probabilistic result can be stated:

\begin{thm} \label{thm9} {\rm (\cite{Luca})} Given a PSD matrix $M \in \mathbb{S}^n$, and  $\varepsilon_\circ  > 0$, let
$u \in \{-1,1\}^n$ be randomly selected with uniform probability. Then  starting with $x_\circ=u/\Vert u \Vert $, the $k$-th iterate of the power method, a unit norm vector $x_k$, with probability at least $3/16$  computes approximation $\lambda_k$ to the largest eigenvalue $\lambda_{\rm max}(M)$, satisfying
\begin{equation}  \label{eeq28}
\lambda_k=x_k^TMx_k \geq  \lambda_{{\rm max}}(M) (1- \varepsilon_\circ) \frac{1}{1+4n (1-\varepsilon_\circ)^{2k}}.
\end{equation}
In particular, when $k = O(\ln n/\varepsilon_\circ)$,  $\lambda_k \geq \lambda_{{\rm max}}(M) (1- O(\varepsilon_\circ))$. \qed
\end{thm}

In summary, the test for a pivot in the Triangle Algorithm is not expected  to  be computationally intensive, especially because the accuracy $\varepsilon_\circ$ in the above theorem for computing a pivot in each iteration of the power method is larger than the accuracy $\varepsilon$ in the Triangle Algorithm.  Indeed we would expect that $\varepsilon_\circ$  would be bounded away from zero over most of the iteration of TA. The following remark suggests we may actually not even need to compute a pivot via the power method too often.

\begin{remark}  \label{rem5} In SHM, as in the case of CHM, a pivot can be used over and over.  That is, while it is easy to show that a pivot used at the current iteration cannot not serve as a pivot at the next iteration, it may well be the case that at the termination of the next iteration or thereafter it can serve as a pivot again.  Next we describe a scheme for employing the generated Euclidean points in $SH(\mathbf{S})$ as much as possible. That is, when possible we treat SHM as if we are solving a CHM.  Note that by the semidefinite Carath\'eodory theorems if $b \in SH(\mathbf{S})$, only $\min \{m+1, n\}$ and even  $O(\sqrt{m})$ extreme points of $\mathbf{\Delta}_n$ are sufficient to express it.
\end{remark}

\subsection{A CHM-Based Triangle Algorithm for SHM} \label{subsec6.2}
Here we describe a strategy in the Triangle Algorithm for solving SHM which basically attempts to run as if it is solving a standard CHM with points in $\mathbb{R}^m$.  In particular, when $m=O(n)$ it would help make the algorithm faster.
Let us assume that we have implemented the Spectrahull Triangle Algorithm, having generated pivots using the power method as described previously.  We thus have generated the following three sets, $U$ a subset of points in the unit sphere in $\mathbb{R}^n$,  $\mathbf{V}$ the set of corresponding rank-one matrices and $S(\mathbf{V})$ the image of $\mathbf{V}$ under the mapping $p(X)=(A_1 \bullet X, \dots, A_m \bullet X)$. Specifically,
\begin{equation}  \label{eeq29}
U= \big \{v_1, \dots, v_t \big \} \subset \mathbb{R}^n, \quad
\mathbf{V}= \big \{V_1=v_1v_1^T, \dots, V_t=v_tv_t^T \big \} \subset \mathbb{S}^n,
\quad S(\mathbf{V})= \big \{p(V_1), \dots, p(V_t) \big\} \subset \mathbb{R}^m.
\end{equation}
Suppose we have stored $U$ and $S(\mathbf{V})$.  Assume the current iterate is $p(X')$.
We can run the Triangle Algorithm to solve the CHM that tests if $b \in conv(S(\mathbf{V}))$, the convex hull $S(\mathbf{V})$. Starting with $p(X')$, as long as we can compute a pivot in $S(\mathbf{V})$, we reduce the gap $\Vert p(X')- b \Vert$  and get closer to $b$ as if we are solving
a standard CHM.   If we compute a desired approximate solution, then we can turn it into a desired approximate solution for SHM.
This is because for any iterate  $p' \in conv(S(\mathbf{V}))$, given the coefficients in its representation as a convex combination, we also have a point $V' \in \mathbf{\Delta}_n$ as a convex combination of points in $\mathbf{V}$ with $p'=p(V')$, using the same coefficients.   However, if $b$ is not found to be in $conv(S(\mathbf{V}))$, as certified by a witness in $conv(S (\mathbf{V}))$, we proceed to find a pivot in $\mathbf{\Delta}_n$ via the power method as  described previously.  Once we generate such a pivot, say $V_{t+1}=v_{t+1}v_{t+1}^T$, we add $v_{t+1}$ to $U$ and  add $p(V_{t+1})$ to $S(\mathbf{V})$, then return to solving the CHM problem. Adding a new pivot may be sufficient to solve the SHM by continuing to iterate in the CHM. If not, it helps reduce the gap as long as we can find a pivot in $S(\mathbf{V})$.

In summary, we  work with a CHM problem whose points are being generated one by one but only return to searching for a pivot in $\mathbf{\Delta}_n$ when the search for a pivot in $S(\mathbf{V})$ has failed to produce one.  The underlying  sets $U$  and $S(\mathbf{V})$ will continue to grow, hence not only more space is needed to store these sets, the search for a pivot may also grow as a function $t$, the cardinality of these sets. Nevertheless by the  semidefinite Carath\'eodory theorems, the current iterate can be written as a convex combination of $O(\min \{m,n\})$  of the $V_i$'s. The computational complexity in the reduction of representation of a point, written as a convex combination of $t$ points, into one as convex combination of at most $O(m)$ can be shown to be $O(m^2t)$. Such reduction in representation can be implemented after every so many iterations. In any case we may state the following complexity result.

\begin{thm}  Consider the Triangle Algorithm for SHM, testing if $b \in SH(\mathbf{S})$ by searching for an $\varepsilon$-approximate solution. Let the set of pivots generated by the algorithm be $\mathbf{V}=\{V_1, \dots, V_T\}$, where $V_i=v_iv_i^T$, $\Vert v_i \Vert =1$, $i=1, \dots, T$.  Let $S(\mathbf{V})=\{p(v_1), \dots, p(v_T)\}$. Ignoring the complexity of the power method needed to compute $\mathbf{V}$, the overall complexity of the Spectrahull Triangle Algorithm is $O(mn^2/\varepsilon^2)$. Alternatively,  the overall complexity of the CHM-based Triangle Algorithm is
\begin{equation} \label{eqcomp}
O\bigg (mn^2T+ mN^2+ \frac{N}{\varepsilon^2} \bigg ),
\end{equation}
for some $N \leq T=O(1/\varepsilon^2)$, where $N$ is the cardinality of the subset of $S(\mathbf{V})$ used as CHM pivots.
\end{thm}
\begin{proof} Given the definition of $T$, it takes $O(mn^2T)$ to generate $\mathbf{V}$. The remaining complexity in (\ref{eqcomp}) is according to Theorem \ref{thmCHMX}, where $C=conv(S(\mathbf{V}))$.
\end{proof}

\begin{remark} If $T=O(n)$, the overall complexity becomes $O(mn^3+ n/\varepsilon^2)$, nearly analogous to that of Triangle Algorithm CHM.
\end{remark}

If the size of the set $S(\mathbf{V})$ grows to be large we can reduce it to the subset of vertices of the convex hull of $S(\mathbf{V})$, or an approximation to it. This task can be accomplished efficiently via AVTA \cite{AKZ}.  The advantage in the CHM-based approach described is that we continue to work with a finite subset of $SH(\mathbf{S})$ in order to test if $b$ lies in $SH(\mathbf{S})$.  The overall number of iterations in this approach may remain to be within the same complexity factor as before, however each iteration will be much more efficient because on the one hand it works in dimension $m$, especially if $m=O(n)$. On the other hand, it gains from all insights accumulated with respect to CHM, e.g. the {\it Spherical-TA}, see \cite{kalsphere}, which scales the given set of points, in this case $S(\mathbf{V})$,  so that the query point $b$ is within the same distance from  the scaled points. The scaling gives a CHM equivalent to the unscaled case but gives geometric insights and measurable conditions to help make the Triangle Algorithm more efficient.  These strategies will be tested computationally in future work.

\subsection{Solving SDP Relaxation of MAX CUT via the Triangle Algorithm} \label{subsec6.3}
Here we consider an application of the Spectrahull Triangle Algorithm. Consider the MAX CUT problem for a given undirected graph $G=(V, E)$ with $|V|=n$, where each edge $(i,j)$ has a nonnegative integer weight $w_{ij}$. The problem is to find a subset $S$ of $V$ that maximizes $ \sum_{i \in S, j \in V \setminus S} w_{ij}$.
Equivalently, this can be determined by maximizing $0.5 \sum_{(i,j) \in E} w_{ij} (1 - x_ix_j)$, where $x \in \{-1, 1\}^n$. Let $W=(w_{ij})$ be the $n \times n$ matrix where $w_{ij}=0$ for all $(i,j) \not \in E$,   $w_{ij}=w_{ji}$. Letting $Y=xx^T$ and ignoring the constant multiplier, the optimization problem of interest and the SDP relaxation in  Goemans and Williamson's formulation \cite{GW}, are respectively
\begin{equation}  \label{eeq31}
\min \big \{W \bullet Y :  Y = xx^T, x \in \{-1, 1\}^n \big \}, \quad \min \big \{W \bullet Y :  E_i \bullet Y=1, i=1, \dots, n, Y \succeq 0  \big \},
\end{equation}
where  $E_i={\rm diag}(e_i)$, $e_i$ the $i$-th standard basis, thus $E_i \bullet Y= Y_{ii}$.  We can solve the SDP relaxation as a sequence of SDP feasibility problems based on binary search on the range of  optimal objective value. Let $X= Y/n$.
Since $Tr(X)=1$, from Proposition \ref {prop10} we can write $|W \bullet X |\leq \Vert W \Vert_2$. Thus the  optimal objective value of SDP relaxation lies
in the interval $[-n\Vert W \Vert_2, n\Vert W \Vert_2]$.
We see that for a given $w$ in this range, the problem of interest is an SHM, where
\begin{equation}  \label{eeq32}
b= \frac{1}{n}(w, 1, \dots, 1)^T \in \mathbb{R}^{n+1}, \quad SH(\mathbf{S})= \big \{p(X)=(W \bullet X, E_1 \bullet X, \dots, E_n \bullet X)^T:  X \in \mathbf{\Delta}_n \big \} \subset \mathbb{R}^n.
\end{equation}
From (\ref{eeq25}) and Proposition \ref{prop10} we may conclude that the corresponding $R =O(\Vert W \Vert_2)$.  Consider testing via TA if in the SHM of (\ref{eeq32}) there exists $X \in \mathbf{\Delta}_n$, satisfying  $\Vert p(X)- b \Vert \leq \varepsilon/n$. If this is solvable, $\Vert p(Y)- nb \Vert \leq \varepsilon$.
Based on Theorem \ref{thm2} the worst-case number of iterations is $O(R^2n^2/\varepsilon^2)$.  However, under some theoretical assumptions, see \cite{kalsphere}, we would expect its iteration complexity to be  $O(Rn/\varepsilon)$, when SHM is feasible, and generally much better when infeasible. This would imply in solving SHM of (\ref{eeq32}) we do not need to solve the infeasible ones to accuracy stated above as these would find a witness quickly.  These suggest the number of iterations to solve MAX CUT via the Triangle Algorithm for the underlying SHM to be within the following bounds (the logarithmic factor coming from binary search in the objective function range):
\begin{equation}  \label{eeq33}
O \bigg ( \frac{n}{\varepsilon} \Vert W \Vert_2 \ln (n\Vert W \Vert_2) \bigg), \quad  O \bigg ( \frac{n^2}{\varepsilon^2} \Vert W \Vert_2^2 \ln (n \Vert W \Vert_2) \bigg).
\end{equation}
Since $E_i$'s are simple the complexity of updating an iteration is $O(|E|)$. Assuming  the complexity of computing a pivot via the power method takes on the average between $O(|E|)$ to $O(|E| \ln n)$ (see Theorem \ref{thm9}) operations, we have to multiply the iteration bounds in (\ref{eeq33}) with this complexity. However, recall from (\ref{subsec6.2}) that TA for general SHM allows interactions with an underlying CHM.  Those iterations with pruning can be brought to within $O(n)$ complexity per iteration.  These are rough analysis of complexity of solving SDP relaxation of MAX CUT. The dependence on $R$ and hence $\Vert W \Vert_2$ may well be over-exaggerated in this analysis. Also,
given that in practice $\varepsilon$  does not need to be too small, the above complexities could make the use of TA competitive with polynomial time interior point algorithms. Many strategies are possible for solving MAX CUT as a sequence of SHM problems, e.g. in the course of binary search, going from one problem to the next one can use a witness from a previous iteration in order to solve the current problem. Solving MAX CUT via TA and these strategies are well worthy of examination and we will consider them in future work.

\section{A Triangle Algorithm for SVM Version of SHM} \label{sec7}

A more general version of SHM is {\it SHM-SVM} defined as follows.
Given two subset of $\mathbb{S}^n$ , $\mathbf{S}=\{A_1, \dots A_m\}$, $\mathbf{S'}=\{A'_1, \dots A'_{m'}\}$, consider their  spectrahulls
$C=SH(\mathbf{S})$, $C'=SH(\mathbf{S'})$. We wish to test if
$C,C'$  intersect and if so to compute an approximate intersection point.  If disjoint, we wish to compute
a separating hyperplane, or estimate the distance between them, or compute an approximate optimal pair of supporting hyperplanes.  When each of the underlying matrices are diagonal matrices, the corresponding problems  reduce to the Euclidean hard-margin {\it Support Vector Machines} (SVM), a problem in machine learning, see e.g. \cite{Burges, Vapnik}. For application of TA to SVM, as well as computational results, see  \cite{Gupta},  where it shows favorable performance in comparison  with the SMO algorithm (see \cite{Platt}).

Let $\delta_* = d(C,C')$ denote the Euclidean distance between $C, C'$, $\rho_*$  the maximum of their diameters, and $\varepsilon$ a prescribed tolerance. The Triangle Algorithm  described in \cite{kalsep} together with the SHM version described here can be used to approximate $\delta_*$, or induce a separating hyperplane, or approximate optimal supporting hyperplanes. Specifically, we can describe a version of the Triangle Algorithm to compute $(p,p') \in C \times C'$ satisfying any of the following desired conditions, when applicable:

\noindent {(1)} $d(p,p') \leq \varepsilon d(p,v)$, $v \in C$, or $d(p,p') \leq \varepsilon d(p',v')$, $v' \in C'$ (when $\delta_*=0$);

\noindent {(2)} the orthogonal bisector of $pp'$ separates $C$ from $C'$;

\noindent {(3)}  $d(p,p') - \delta_* \leq \varepsilon d(p,p')$ (when $\delta_* >0$);

\noindent {(4)} a pair of supporting hyperplanes $(H,H')$ to $C,C'$ orthogonal to $pp'$ satisfy $\delta_* -d(H,H') \leq \varepsilon d(p,p')$.

From the results in \cite{kalsep} and those stated in this article we  conclude the following iteration complexity bounds for solving the above four problems. In particular, when $C'= \{b\}$, a single point, tasks (1) and (2) are described in this article.  However, we can also estimate the distance from $b$ to $C$ to  prescribed accuracy. The algorithm for that is analogous with that described in \cite{kalsep}, however it has to be combined with  the results developed in this article.  We avoid such details.  The following complexity bound can be deduced.

\begin{thm} \label{thm10}
The corresponding number of iterations to solve tasks (1)-(4) are respectively, $O(1/\varepsilon^2)$, $O(\rho_*^2/\delta_*^2)$, and $O(\rho_*^2/\delta_*^2 \varepsilon)$ for the last two. The complexity in each iteration of the first two tasks is computing for a given pair of iterates $(p,p') \in C \times C'$ a {\it pivot},  i.e. $v \in C$ with $d(p,v) \geq d(p',v)$, or $v' \in C'$ with $d(p',v') \geq d(p,v')$. For the last two tasks the complexity of each iteration is either computing a pivot, or a pair of supporting  hyperplanes $(H,H')$  orthogonal to $pp'$.  As in SHM, in the worst-case the complexity of each iteration is running the power method to estimate a least eigenvalue of a symmetric matrix and a corresponding eigenvector. \qed
\end{thm}

\section*{Concluding Remarks and Future Work}
Given a subset $\mathbf{S}=\{A_1, \dots, A_m \}$ of  $n \times n$ real symmetric matrices, $\mathbb{S}^n$, we defined its {\it spectrahull}, denoted by $SH(\mathbf{S})$, to be
an analogue of the convex hull of a finite set of points in a Euclidean space.
Then given a point $b \in \mathbb{R}^m$,
we defined {\it spectrahull membership} (SHM) to be the problem of testing if $b \in SH(\mathbf{S})$, a semidefinite version of the {\it convex hull membership} (CHM), defined over the spectraplex. In fact when $A_i$'s are diagonal matrices the spectraplex can be replaced with the unit simplex and  the spectrahull and SHM reduce to the ordinary convex hull and CHM, respectively.

An SDP feasibility problem having no recession direction is reducible to a homogeneous SHM. Hence any algorithm for SHM is applicable to such SDP feasibility problem and more generally to any SDP optimization over a bounded feasible set.  By proving  $SH(\mathbf{S})$ is compact and convex, we reduced SHM to a special case of
{\it General-CHM}, the  problem of testing if a point $b \in \mathbb{R}^m$ belongs to a given compact convex subset $C$ of $\mathbb{R}^m$. Next we developed  a version of the {\it Triangle Algorithm} in \cite{kalsep} in order to solve SHM. This required proving a {\it distance duality} for SHM, also showing that the task of computing a pivot is equivalent to the estimation of an eigenvector corresponding to the least eigenvalue of a symmetric matrix arising in each iteration. Testing for a pivot  can thus be carried out via a modified power method, not typically needed to be executed to high precision. The reduction of the search for a pivot to an eigenvalue estimation was justified by proving that if a pivot exists at an iterate, a rank-one pivot must necessarily  exist.  Since the power method is the main computational tool of the algorithm and easy to implement, we would expect each iteration of TA for SHM would be  easy and flexible to implement. This is analogous to its CHM counterpart.
Not only that, as in the case of CHM, we would expect that some of the computed  pivots  would be utilized over and over at later stages. This suggests other strategies in solving SHM via the Triangle Algorithm. For example,  storing the base unit vectors, $U$ that define the pivots and test if $b$ lies in the convex hull of the corresponding images $S(\mathbf{V})$ (see (\ref{eeq29})), a subset of $SH(\mathbf{S})$, as if  solving a CHM.  When no pivot would be available, as certified by a witness for the corresponding CHM, then and only then we would generate a new pivot using the power method. This makes solving SHM to be treated  as close as possible to solving an ordinary CHM. We also stated a potentially more efficient complexity for this CHM-based algorithm for SHM. Thus TA for SHM can allow substantial interplay with an inherent CHM.   Additionally,  we can improve on storage space by storing only the subset of $U$ whose corresponding points in $S(\mathbf{V})$ are the vertices of $conv(S(\mathbf{V}))$. This task can be accomplished using {\it AVTA} \cite{AKZ}, an  algorithm for finding all vertices, or a good subset of vertices of  points in the Euclidean space. This strategy as well as working with the underlying CHM are also supported by the semidefinite versions of Carath\'edoroy theorem stated in the article.

In summary, the proposed Triangle Algorithm for SHM developed here not only offers theoretical insight on SDP  itself but  it will serve as an alternative to the existing algorithms for solving bounded SDP-Feasibility, as well as general SDP optimization problems. In particular, we showed how the SDP relaxation of MAX CUT can be solved via the TA for SHM.  We will carry out computational experimentation with TA for solving SDP relaxation of this and other combinatorial problems and will report on the results.  We also described a semidefinite version of SVM, where it can be solved as the Euclidean version of SVM via the  Triangle Algorithm described in \cite{kalsep} but using the machinery developed here for SHM. Regardless of the potential practical utility of SHM-SVM, the ability to solve it via the Triangle Algorithm is indicative of the power of the algorithm. In fact the definition of spectrahull and SHM can be extended to a more general case, where $\mathbb{S}^n$ and the cone of positive semidefinite matrices are replaced with a general Hilbert space and
an appropriate pointed cone, respectively.  Thus one can define a corresponding  Triangle Algorithm and extend the theoretical iteration complexity bounds.  Returning to SHM,  it is also possible to define more general problems.  For example, given $\mathbf{S}=\{A_1, \dots, A_m \}$ as before,  consider the set  $\{X \in \mathbf{\Delta}_n:  q_i(A_i \bullet X), i=1, \dots, m\}$, where each $q_i(x)$ is a polynomial in $x$, say a quadratic polynomial.
Consider for example a generalization of quadratically constrained quadratic programming.  We may wish to test if a given $b \in \mathbb{R}^m$ lies in the convex hull of this set.  We can still employ the TA but of course the computation of a pivot would become more complicated. In future work we will carry out computational results with TA on some SDP feasibility problems, solving them as SHM. As with the convex hull of a finite point set, we anticipate that consideration of  spectrahull of a finite subset of symmetric matrices will give rise to new problems, e.g. the notion of irredundancy and approximation of a spectrahull via a subset of its extreme points. Finally, in a forthcoming article we will offer yet another version of the Triangle Algorithm for solving the standard LP and SDP feasibility problems.

\bigskip


\end{document}